\documentclass[12pt]{article}
\usepackage{amsfonts}

\usepackage[all]{xy}
\usepackage{amsmath}

\def\squarebox#1{\hbox to #1{\hfill\vbox to #1{\vfill}}}
\newcommand{\qed}{\hspace*{\fill}
\vbox{\hrule\hbox{\vrule\squarebox{.667em}\vrule}\hrule}\smallskip}
\newtheorem{teorema}{Theorem}[section]

\newtheorem{corolario}[teorema]{Corollary}
\newtheorem{proposicao}[teorema]{Proposition}

\newenvironment{proof}{\noindent {\bf Proof:}}{\hfill $\qed $ \newline}

\begin{document}

\title{The Topological Entropy \\ of Powers on Lie Groups}
\author{Mauro Patr\~{a}o\footnote{
    Department of Mathematics, University of Bras\'ilia, Brazil. 
    mpatrao@mat.unb.br
  }}
\maketitle

\begin{abstract}
This article addresses the problem of computing the topological entropy of an application $\psi : G \to G$, where $G$ is a Lie group, given by some power $\psi(g) = g^k$, with $k$ a positive integer. When $G$ is commutative, $\psi$ is an endomorphism and its topological entropy is given by $h(\psi) = \dim(T(G)) \log(k)$, where $T(G)$ is the maximal torus of $G$, as shown in \cite{patrao:endomorfismos}. But when $G$ is not commutative, $\psi$ is no longer an endomorphism and these previous results cannot be used. Still, $\psi$ has some interesting symmetries, for example, it commutes with the conjugations of $G$. In this paper, the structure theory of Lie groups is used to show that $h(\psi) = \dim(T)\log(k)$, where $T$ is a maximal torus of $G$, generalizing the commutative case formula. In particular, the topological entropy of powers on compact Lie groups with discrete center is always positive, in contrast to what happens to endomorphisms of such groups, which always have null entropy.
\end{abstract}



\section{Introduction}

The computation of the topological entropy of a continuous endomorphism $\phi$ of a Lie group $G$ is a classical topic in ergodic theory which seemed to have long been solved. Recently, in \cite{patrao:endomorfismos}, we proved that, for a continuous endomorphism $\phi$ of an arbitrary Lie group $G$, we have that
\begin{equation}\label{eq-introduction-endomorphism}
h\left(\phi\right) = h\left(\phi|_{T(G_\phi)}\right)
\end{equation}
where $G_\phi$ is the maximal connected subgroup of $G$ such that $\phi(G_\phi) = G_\phi$, $T(G_\phi)$ is the maximal torus in the center of $G_\phi$, and the topological entropy is the natural generalization for locally compact metric spaces introduced in \cite{patrao:entropia, caldas-patrao:entropia}, which is characterized by the so-called variational principle (see Proposition \ref{theorem-3-1-caldas-patrao:entropia}). An immediate consequence of the above equation is that $h\left(\phi\right) = 0$ for every continuous endomorphism $\phi$ of a semi-simple Lie group $G$. In particular, this is true for continuous endomorphisms of a compact simple Lie group, as the multiplicative group of quaternions of norm one $G = S^3$. So it would be interesting to determine the topological entropy of a class of maps on Lie groups such that its maps have positive topological entropy when $G$ is a compact Lie group.

This article addresses the problem of computing the topological entropy of an application $\psi : G \to G$, where $G$ is a Lie group, given by some power $\psi(g) = g^k$, with $k$ a positive integer. When $G$ is commutative, $\psi$ is an endomorphism and it follows from the equation (\ref{eq-introduction-endomorphism}) and from the well known formula for the topological entropy of an endomorphism of a torus (see Proposition \ref{corollaries-11-16-bowen-endomorphis}) that its topological entropy is given by $h(\psi) = \dim(T(G)) \log(k)$, where $T(G)$ is the maximal torus of $G$. For example, if $G = \mathbb{C}^*$ and $\psi(z) = z^2$, it is easy to see that $T(G) = S^1$, which implies that $h\left(\psi\right) = \log(2)$. On the other hand, if $G = S^3$ and $\psi(g) = g^2$, what is the value of $h\left(\psi\right)$?

When $G$ is a compact and connected Lie group, we have the following structural result (see Proposition \ref{corollary-12-2-11-hilgert-neeb-theorem-12-2-2-hilgert-neeb})
\begin{equation}\label{eq-introduction-compact-torus}
G = \bigcup_{g \in G} gTg^{-1}
\end{equation}
where $T \subset G$ is any fixed maximal torus. Thus we can consider the following commutative diagram
\begin{displaymath}
    \xymatrix{
        G \times T \ar[r]^\Psi \ar[d]_c & G \times T \ar[d]^c \\
        G \ar[r]_\psi      & G }
\end{displaymath}
where $\Psi(g,t) = (g, \psi(t))$ and $c(g,t) = gtg^{-1}$. By equation (\ref{eq-introduction-compact-torus}), the map $c$ is surjective, and since $G$ and $T$ are compact, it follows that
\begin{eqnarray}
h\left(\psi\right) 
& \leq & h\left(\Psi\right) \nonumber \\
& = & h\left(\psi|_T\right)  \nonumber \\
& \leq & h\left(\psi\right)  
\end{eqnarray}
which implies that
\begin{equation}\label{eq-introduction-entropy-power}
h(\psi) = \dim(T) \log(k)
\end{equation}
If $G = S^3$ and $\psi(g) = g^2$, it is easy to see that any maximal torus is isomorphic to $S^1$, which implies that $h\left(\psi\right) = \log(2)$. When $G$ is connected but not compact, the right hand side of equation (\ref{eq-introduction-compact-torus}) is just the union $X$ of all compact subgroups of $G$. On the other hand, when $G$ is compact but not connected, the right hand side of equation (\ref{eq-introduction-compact-torus}) is just the connected component $G_0$ of the identity of $G$. For the noncompact case, we need to show that the recurrent set $\mathcal{R}_{\psi}$ is contained in $X$ and we need to develop a variational principle adapted to $X$, which might not be a locally compact space. For the nonconnected case, we need a new structural result generalizing equation (\ref{eq-introduction-compact-torus}) in a suitable way.

The paper is organized as follows. In Section \ref{Ergodic-Theory}, we collect the main concepts and results, and prove some new ones, about ergodic theory that are used in the final section. In particular, the variational principle proved in Theorem \ref{theo-variational-principle} is interesting by itself and can be used in a much broader context. In Section \ref{Lie-Theory}, we collect the main concepts and results, and prove some new ones, about Lie theory that are used in the final section. In particular, the structural results proved in Theorem \ref{theo-conjugation-normalizer-torus} and Corolllary \ref{cor-conjugation-normalizer-torus} are also interesting by themselves. Finally, in Section \ref{topological}, we prove the main result of this paper, Theorem \ref{theo-entropy-powers}, which shows that equation (\ref{eq-introduction-entropy-power}) remains true for arbitrary Lie groups.

\section{Preliminaries on ergodic theory}\label{Ergodic-Theory}

In this section, we collect the main concepts and results, and prove some new ones, about ergodic theory that are used in the final section. Given a set $X$, a family $\mathcal{C}$ of subsets of $X$ is a cover of $X$ when
\begin{equation}
X = \bigcup_{C \in \mathcal{C}} C
\end{equation}
If the sets in $\mathcal{C}$ are disjoint, then we say that $\mathcal{C}$ is a partition of $X$. A subcover
of $\mathcal{C}$ is a family $\mathcal{D} \subset \mathcal{C}$ which is itself a cover of $X$. We denote by $N(\mathcal{C})$ the least cardinality amongst the subcovers of $\mathcal{C}$. 

Consider the topological space $X$, a Radon probability measure $\mu$ on $X$, and finite measurable partition $\mathcal{P}$. The partition entropy of $\mathcal{P}$ is given by
\begin{equation}
H_\mu\left(\mathcal{P}\right) = \sum_{P \in \mathcal{P}} \mu\left(P\right)\log\frac{1}{\mu\left(P\right)}
\end{equation}
Let $\psi : X \to X$ be a $\mu$-invariant map. The partition entropy of $\psi$ with respect to $\mathcal{P}$ is given by
\begin{equation}
h_\mu\left(\psi, \mathcal{P}\right) = \lim_{j \to \infty} \frac{1}{j} H_\mu\left(\mathcal{P}^j\right)
\end{equation}
where $\mathcal{P}^j$ is the finite measurable partition given by the subsets
\begin{equation}
P_0 \cap \psi^{-1}\left(P_1\right)  \cap \cdots \cap \psi^{-(j-1)}\left(P_{j-1}\right)
\end{equation}
with $P_0, P_1, \ldots, P_{j-1} \in \mathcal{P}$. The Kolmogorov-Sinai entropy of $\psi$ is given by
\begin{equation}
h_\mu\left(\psi\right) = \sup_{\mathcal{P}} h_\mu\left(\psi, \mathcal{P}\right)
\end{equation}
where the supremum is taken over all finite measurable partitions $\mathcal{P}$. Given another finite measurable partition $\mathcal{Q}$, the conditional entropy is defined as the expected value
\begin{equation}
H_\mu\left(\mathcal{Q}|\mathcal{P}\right) = \sum_{C \in \mathcal{P}} \mu\left(\mathcal{P}\right)H_{\mu\left(\cdot|\mathcal{P}\right)}\left(\mathcal{Q}\right)
\end{equation}

Now an open cover $\mathcal{A}$ is called admissible if at least one of its elements has compact complement and it is called strongly admissible if all of its elements have compact complement. The cover
entropy of $\mathcal{A}$ is given by 
\begin{equation}
H\left(\mathcal{A}\right) = \log N\left(\mathcal{A}\right)
\end{equation}
Let $\psi : X \to X$ be a continuous map. The topological entropy of $\psi$ with respect to $\mathcal{A}$ is
\begin{equation}
h\left(\psi,\mathcal{A}\right) = \lim_{j \to \infty} \frac{1}{j} H\left(\mathcal{A}^j\right)
\end{equation}
where $\mathcal{A}^j$ is the open cover given by the subsets
\begin{equation}
A_0 \cap \psi^{-1}\left(A_1\right)  \cap \cdots \cap \psi^{-(j-1)}\left(A_{j-1}\right)
\end{equation}
with $A_0, A_1, \ldots, A_{j-1} \in \mathcal{A}$. The topological entropy of $\psi$ is given by
\begin{equation}
h\left(\psi\right) = \sup_{\mathcal{A}} h\left(\psi, \mathcal{A}\right)
\end{equation}
where the supremum is taken over all admissible open covers $\mathcal{A}$. The proof following of the proposition can be found in Lemmas 2.6 and 2.12 of \cite{caldas-patrao:entropia}.

\begin{proposicao}\label{lemma-2-6-caldas-patrao:entropia-lemma-2-12-caldas-patrao:entropia}
If $(Y, \mu)$ is a probability space and $\mathcal{C}$ is a finite measurable partition, then
\begin{equation}
H_\mu\left(\mathcal{C}\right) \leq \log N\left(\mathcal{C}\right)
\end{equation}
If $\psi : Y \to Y$ is a $\mu$-invariant map and $\mathcal{D}$ is another finite measurable partition, then
\begin{equation}
h_\mu\left(\psi, \mathcal{C}\right) \leq h_\mu\left(\psi, \mathcal{D}\right) +  H_\mu\left(\mathcal{D}|\mathcal{C}\right)
\end{equation}
\end{proposicao}

The proof of the following result can be found in Proposition 2.1.8 of \cite{ferraiol-dissertacao} and Remark 2.20 of \cite{caldas-patrao:entropia}.

\begin{proposicao}\label{proposition-2-1-8-ferraiol-dissertacao-remark-2-20-caldas-patrao:entropia}
Let $(Y, \mu)$ be a probability space and $\psi : Y \to Y$ be a $\mu$-invariant map. If $X \subset Y$ is a $\psi$-invariant measurable subset, then 
\begin{equation}
h_\mu\left(\psi\right) = \mu\left(X\right)h_{\mu_X}\left(\psi|_X\right) + \mu\left(Y\backslash X\right)h_{\mu_{Y\backslash X}}\left(\psi|_{Y\backslash X}\right)
\end{equation}
where $\mu_X$ and $\mu_{Y\backslash X}$ are the probability measures conditional to $X$ and $Y\backslash X$. Furthermore, we have that
\begin{equation}
h_\mu\left(\psi^n\right) = n h_\mu\left(\psi\right)
\end{equation}
for each positive integer $n$.
\end{proposicao}

The proof of the following proposition can be found in Lemma 2.19 of \cite{caldas-patrao:entropia}.

\begin{proposicao}\label{lemma-2-19-caldas-patrao:entropia}
Let $X$ be a topological space and $\psi : X \to X$ be a continuous map. Then
\begin{equation}
h\left(\psi^n\right) \leq n h\left(\psi\right)
\end{equation}
for each positive integer $n$.
\end{proposicao}

The proof of the following result can be found in Theorem 3.1 of \cite{caldas-patrao:entropia}.

\begin{proposicao}\label{theorem-3-1-caldas-patrao:entropia}
Let $Y$ be a metrizable locally compact separable topological space and $\psi : Y \to Y$ be a continuous map. Then
\begin{equation}
\sup_\mu h_\mu\left(\psi\right) = h\left(\psi\right)
\end{equation}
where the supremum is taken over all $\psi$-invariant Radon probability measures $\mu$ on $Y$.
\end{proposicao}

The proof of the following proposition can be found in Corollaries 11 and 16 of \cite{bowen-endomorphism}.

\begin{proposicao}\label{corollaries-11-16-bowen-endomorphis}
Let $T$ be a torus, $\varphi : T \to T$ be a continuous endomorphism and $L_g: T \to T$ be the map given by $L_g(h) = gh$. Then
\begin{equation}
h\left(L_g \circ \varphi\right) = h\left(\varphi\right) = \sum_{\lambda} \log |\lambda|
\end{equation}
where the summation is taken over all eigenvalues $\lambda$ of $d\varphi_1$ such that $|\lambda| > 1$.
\end{proposicao}

We also need the following proposition.

\begin{proposicao}\label{prop-disjoint-union}
Let $Y$ be a metrizable locally compact separable topological space and $\psi: Y \to Y$ be a continuous map. If 
\begin{equation}
Y = Y_1 \cup \cdots \cup Y_j
\end{equation}
where $\{Y_1, \ldots, Y_j\}$ is a family of disjoint $\psi$-invariant locally compact subsets, then
\begin{equation}
h\left(\psi\right) = \max_{i = 1,\ldots,j} h\left(\psi|_{Y_i}\right)
\end{equation}
\end{proposicao}

\begin{proof}
By Proposition \ref{proposition-2-1-8-ferraiol-dissertacao-remark-2-20-caldas-patrao:entropia}, for each $\psi$-invariant Radon probability measure $\mu$, we have that
\begin{eqnarray}
h_\mu\left(\psi\right) 
& = & \mu\left(Y_1\right)h_{\mu_{Y_1}}\left(\psi|_{Y_1}\right) + \cdots + \mu\left(Y_j\right)h_{\mu_{Y_j}}\left(\psi|_{Y_j}\right) \nonumber \\
& \leq & \mu\left(Y_1\right)h\left(\psi|_{Y_1}\right) + \cdots + \mu\left(Y_j\right)h\left(\psi|_{Y_j}\right) \nonumber \\
& \leq & \max_{i = 1,\ldots,j} h\left(\psi|_{Y_i}\right)
\end{eqnarray}
where we used Proposition \ref{theorem-3-1-caldas-patrao:entropia} in first inequality. Hence
\begin{eqnarray}
h\left(\psi\right) 
& = & \sup_{\mu} h_\mu\left(\psi\right) \nonumber \\
& \leq & \max_{i = 1,\ldots,j} h\left(\psi|_{Y_i}\right) \nonumber \\
& \leq & h\left(\psi\right)
\end{eqnarray}
\end{proof}

Now we prove the following variational principle which is crucial for the computation of the topological entropy of powers when $G$ is noncompact.

\begin{teorema}\label{theo-variational-principle}
Let $X$ be a metric space and $\psi: X \to X$ be a continuous map. If
\begin{equation}
X = \bigcup_{l = 0}^\infty C_l
\end{equation}
where each $C_l$ is a $\psi$-invariant compact subset and $C_l \subset C_{l+1}$ for each $l$, then
\begin{equation}
\sup_{\mu} h_\mu\left(\psi\right) = \lim_{l \to \infty} h\left(\psi|_{C_l}\right)
\end{equation}
where the supremum is taken over all $\psi$-invariant Radon probability measures $\mu$ on $X$.
\end{teorema}

\begin{proof}
In order to prove that 
\begin{equation}
\sup_{\mu} h_\mu\left(\psi\right) \leq \lim_{l \to \infty} h\left(\psi|_{C_l}\right)
\end{equation}
it is sufficient to show that, given a $\psi$-invariant Radon probability measure $\mu$, there exists $l$ such that
\begin{equation}\label{eq-variational-principle-1}
h_\mu\left(\psi\right) \leq h\left(\psi|_{C_l}\right)
\end{equation}
since 
\begin{equation}
h\left(\psi|_{C_l}\right) \leq \lim_{l \to \infty} h\left(\psi|_{C_l}\right)
\end{equation}
In fact, it is enough to show that there exists $l$ such that
\begin{equation}\label{eq-variational-principle-2}
h_\mu\left(\psi^n\right) \leq h\left(\psi|_{C_l}^n\right) + 2 + \log(2)
\end{equation}
for each natural number $n$, since
\begin{eqnarray}
h_\mu\left(\psi\right)
& = & \frac{1}{n} h_\mu\left(\psi^n\right) \nonumber \\
& \leq & \frac{1}{n} h\left(\psi|_{C_l}^n\right) + \frac{2 + \log(2)}{n} \nonumber \\
& \leq & h\left(\psi|_{C_l}\right) + \frac{2 + \log(2)}{n}
\end{eqnarray}
where we used Proposition \ref{proposition-2-1-8-ferraiol-dissertacao-remark-2-20-caldas-patrao:entropia} in the equality and Proposition \ref{lemma-2-19-caldas-patrao:entropia} in the second inequality. Hence inequality (\ref{eq-variational-principle-1}) follows by taking the limit as $n$ goes to the infinity. In order to show inequality (\ref{eq-variational-principle-2}), take a finite measurable partition $\mathcal{P}$ such that
\begin{equation}\label{eq-variational-principle-3}
h_\mu\left(\psi^n\right) \leq h_\mu\left(\psi^n, \mathcal{P}\right) + 1
\end{equation}
For each $P \in \mathcal{P}$, chose $C_P \subset P$ compact such that
\begin{equation}
\mu\left(P \backslash C_P\right) \leq \frac{1}{2N\left(\mathcal{P}\right) \log N\left(\mathcal{P}\right)}
\end{equation}
where $N\left(\mathcal{P}\right)$ is the cardinality of $\mathcal{P}$. Since $X = \bigcup_{l = 0}^\infty C_l$, $\mu(X) = 1$, and $C_l \subset C_{l+1}$ for each $l$, it follows that
\begin{equation}
\lim_{l \to \infty} \mu\left(C_l\right) = 1
\end{equation}
Thus we can choose $l$ such that 
\begin{equation}
\mu\left(X \backslash C_l\right) \leq \frac{1}{2N\left(\mathcal{P}\right) \log N\left(\mathcal{P}\right)}
\end{equation}
and define
\begin{equation}
C_P^l = C_P \cap C_l
\end{equation}
It follows that
\begin{eqnarray}
\mu\left(P \backslash C_P^l \right)
& = & \mu\left(P \cap (C_P^l)^c \right) \nonumber \\
& = & \mu\left(P \cap (C_P \cap C_l)^c \right) \nonumber \\
& = & \mu\left(P \cap ((C_P)^c \cup (C_l)^c) \right) \nonumber \\
& = & \mu\left( (P \cap (C_P)^c) \cup (P \cap (C_l)^c) \right) \nonumber \\
& \leq & \mu\left(P \cap (C_P)^c\right) + \mu\left(P \cap (C_l)^c\right) \nonumber \\
& \leq & \mu\left(P \backslash C_P \right) + \mu\left(X \backslash C_l \right) \nonumber \\
& \leq & \frac{1}{N\left(\mathcal{P}\right) \log N\left(\mathcal{P}\right)}
\end{eqnarray}
Defining 
\begin{equation}
P^l = \bigcup_{P \in \mathcal{P}} P \backslash C_P^l
\end{equation}
it follows that
\begin{equation}
\mu\left(P^l \right) \leq \frac{1}{\log N\left(\mathcal{P}\right)}
\end{equation}
Define the following measurable partition
\begin{equation}
\mathcal{P}_l = \{ C_P^l : P \in \mathcal{P} \} \cup \{P^l\}
\end{equation}
and the strongly admissible cover
\begin{equation}
\mathcal{A} = \{ C_P^l \cup P^l : P \in \mathcal{P} \}
\end{equation}
We claim that 
\begin{equation}\label{eq-variational-principle-4}
H_\mu\left(\mathcal{P}|\mathcal{P}_l\right) \leq 1
\end{equation}
In fact, first note that $\mu\left(P | C_P^l \right) = 1$ for every $P \in \mathcal{P}$. Thus
\begin{equation}
H_{\mu(.|C_P^l)}\left(\mathcal{P}\right) = 0
\end{equation}
and therefore, by the definition of conditional entropy and by Proposition \ref{lemma-2-6-caldas-patrao:entropia-lemma-2-12-caldas-patrao:entropia}, it follows that
\begin{eqnarray}
H_\mu\left(\mathcal{P}|\mathcal{P}_l\right)
& = & \mu\left(P^l\right)H_{\mu(.|P^l)}\left(\mathcal{P}\right) \nonumber \\
& \leq & \mu\left(P^l\right)\log N\left(\mathcal{P}\right) \nonumber \\
& \leq & 1
\end{eqnarray}
We claim that
\begin{equation}\label{eq-variational-principle-5}
N\left(\mathcal{P}_l^j\right) \leq 2^jN\left(\mathcal{A}^j\right)
\end{equation}
where $\mathcal{A}^j$ is the strongly admissible cover given by the following subsets
\begin{equation}
\left(C_{P_0}^l \cup P^l\right) \cap \psi^{-n}\left(C_{P_1}^l \cup P^l\right)  \cap \cdots \cap \psi^{-n(j-1)}\left(C_{P_{j-1}}^l \cup P^l\right)
\end{equation}
with $P_i \in \mathcal{P}$ for each $i$, and $\mathcal{P}_l^j$ is the measurable partition given by the following subsets
\begin{equation}
Y_0 \cap \psi^{-n}\left(Y_1\right)  \cap \cdots \cap \psi^{-n(j-1)}\left(Y_{j-1}\right)
\end{equation}
where $Y_i = C_{P_i}^l$ or $Y_i = P^l$ for each $i$. Let $m$ be the cardinality of $\mathcal{P}$ and $\Lambda \subset \{1, \ldots, m\}^j$ such that its cardinality is $N\left(\mathcal{A}^j\right)$ and that
\begin{equation}
X = \bigcup_{\lambda \in \Lambda} \left(C_{P_{\lambda_0}}^l \cup P^l\right) \cap \cdots \cap \psi^{-n(j-1)}\left(C_{P_{\lambda_{j-1}}}^l \cup P^l\right)
\end{equation}
where $\lambda = (\lambda_0, \ldots, \lambda_{j-1})$. Consider the map $f: \Lambda \times \{0,1\}^j \to \mathcal{P}_l^j$ given by
\begin{equation}
f(\lambda,x) = Y_0 \cap \psi^{-n}\left(Y_1\right)  \cap \cdots \cap \psi^{-n(j-1)}\left(Y_{j-1}\right)
\end{equation}
with 
\begin{equation}
Y_i 
=
\left\{
\begin{array}{lr}
C_{P_{\lambda_i}}^l, & x_i = 1 \\
P^l, & x_i = 0
\end{array}
\right.
\end{equation}
where $x = (x_0, \ldots, x_{j-1})$. Since 
\begin{equation}
X = \bigcup_{\lambda \in \Lambda, \, x \in \{0,1\}^j }f(\lambda,x)
\end{equation}
and since $\mathcal{P}_l^j$ is a partition, it follows that the image of $f$ contains every nonempty element of $\mathcal{P}_l^j$, which implies the inequality (\ref{eq-variational-principle-5}). Now consider the strongly admissible cover of $C_l$ given by
\begin{equation}
\mathcal{A}_l = \{ \left(C_P^l \cup P_l\right) \cap C_l : P \in \mathcal{P} \}
\end{equation}
We claim that
\begin{equation}\label{eq-variational-principle-6}
N\left(\mathcal{A}^j\right) \leq N\left(\mathcal{A}_l^j\right)
\end{equation}
where $\mathcal{A}_l^j$ is the strongly admissible cover given by the following subsets
\begin{equation}
\left(\left(C_{P_0}^l \cup P^l\right) \cap C_l \right) \cap \cdots \cap \psi^{-n(j-1)}\left(\left(C_{P_{j-1}}^l \cup P^l\right) \cap C_l\right)
\end{equation}
with $P_i \in \mathcal{P}$ for each $i$. Let $\Delta \subset \{1, \ldots, m\}^j$ such that its cardinality is $N\left(\mathcal{A}_l^j\right)$ and that
\begin{equation}
C_l = \bigcup_{\delta \in \Delta} \left(\left(C_{P_{\delta_0}}^l \cup P^l\right) \cap C_l \right) \cap \cdots \cap \psi^{-n(j-1)}\left(\left(C_{P_{\delta_{j-1}}}^l \cup P^l\right) \cap C_l\right)
\end{equation}
where $\delta = (\delta_0, \ldots, \delta_{j-1})$. Consider the map $g: \Delta \to \mathcal{A}^j$ given by
\begin{equation}
g(\delta) = \left(C_{P_{\delta_0}}^l \cup P^l\right) \cap \cdots \cap \psi^{-n(j-1)}\left(C_{P_{\delta_{j-1}}}^l \cup P^l\right)
\end{equation} 
Since $C_l^c \subset P^l$ and since $C_l$ is $\psi^n$-invariant, it follows that $C_l^c \subset g(\delta)$ for each $\delta$ and that
\begin{equation}
C_l = \left(\bigcup_{\delta \in \Delta} g(\delta) \right) \cap C_l
\end{equation} 
Hence 
\begin{equation}
X = C_l \cup C_l^c \subset \bigcup_{\delta \in \Delta} g(\delta)
\end{equation} 
showing inequality (\ref{eq-variational-principle-6}). Taking the logarithm of inequalities (\ref{eq-variational-principle-5}) and (\ref{eq-variational-principle-6}), dividing by $l$ and taking the limit as $l$ tends to infinity, it following that
\begin{eqnarray}\label{eq-variational-principle-7}
h_\mu\left(\psi^n,\mathcal{P}_l\right)
& \leq & h\left(\psi^n,\mathcal{A}\right) + \log(2) \nonumber \\
& \leq & h\left(\psi|_{C_l}^n,\mathcal{A}_l\right) + \log(2) \nonumber \\
& \leq & h\left(\psi|_{C_l}^n\right) + \log(2)
\end{eqnarray}
From inequalities (\ref{eq-variational-principle-3}), (\ref{eq-variational-principle-4}), (\ref{eq-variational-principle-7}), and Proposition \ref{lemma-2-6-caldas-patrao:entropia-lemma-2-12-caldas-patrao:entropia}, it follows that
\begin{eqnarray}
h_\mu\left(\psi^n\right)
& \leq & h_\mu\left(\psi^n,\mathcal{P}\right) + 1 \nonumber \\
& \leq & h_\mu\left(\psi^n,\mathcal{P}_l\right) + H_\mu\left(\mathcal{P}|\mathcal{P}_l\right) + 1 \nonumber \\
& \leq & h_\mu\left(\psi^n,\mathcal{P}_l\right) + 2 \nonumber \\
& \leq & h\left(\psi|_{C_l}^n\right) + \log(2) + 2
\end{eqnarray}
In order to prove that 
\begin{equation}\label{eq-variational-principle-7}
\lim_{l \to \infty} h\left(\psi|_{C_l}\right) \leq \sup_{\mu} h_\mu\left(\psi\right)
\end{equation}
for each $\varepsilon > 0$, there exists $l$ such that
\begin{equation}
\lim_{l \to \infty} h\left(\psi|_{C_l}\right) \leq h\left(\psi|_{C_l}\right) + \frac{\varepsilon}{2}
\end{equation}
and, by the variational principle of entropy for compact spaces, there exists a $\psi|_{C_l}$-invariant Radon probability measure $\mu_l$ such that 
\begin{equation}
h\left(\psi|_{C_l}\right) < h_{\mu_l}\left(\psi|_{C_l}\right) + \frac{\varepsilon}{2}
\end{equation}
Considering the following $\psi$-invariant Radon probability measure $\mu$ given by
\begin{equation}
\mu\left(A\right) = \mu_l\left(A \cap C_l\right)
\end{equation}
it follows that
\begin{equation}
h_\mu\left(\psi\right) = h_{\mu_l}\left(\psi|_{C_l}\right)
\end{equation}
By the previous inequalities, it follows that
\begin{eqnarray}
\lim_{l \to \infty} h\left(\psi|_{C_l}\right)
& < & h_\mu\left(\psi\right) + \varepsilon \nonumber \\
& \leq & \sup_{\mu} h_\mu\left(\psi\right) + \varepsilon
\end{eqnarray}
Since $\varepsilon$ is arbitrary, we obtain inequality (\ref{eq-variational-principle-7}).
\end{proof}

\begin{corolario}\label{cor-variational-principle}
Let $Y$ be a separable metric space and $\psi: Y \to Y$ be a continuous map. If there exists $X \subset Y$ such that $\mathcal{R}_{\psi} \subset X$ and
\begin{equation}
X = \bigcup_{l = 0}^\infty C_l
\end{equation}
where each $C_l$ is a $\psi$-invariant compact subset and $C_l \subset C_{l+1}$ for each $l$, then
\begin{equation}
\sup_{\mu} h_\mu\left(\psi\right) = \lim_{l \to \infty} h\left(\psi|_{C_l}\right)
\end{equation}
where the supremum is taken over all $\psi$-invariant Radon probability measures $\mu$ on $X$.
\end{corolario}

\begin{proof}
For each $\psi$-invariant Radon probability measure $\mu$, we have that
\begin{equation}
h_\mu\left(\psi\right) = \mu\left(X\right)h_{\mu_X}\left(\psi|_{X}\right) + \mu\left(Y \backslash X \right)h_{\mu_{Y \backslash X}}\left(\psi|_{Y \backslash X}\right)
\end{equation}
Since $\mathcal{R}_{\psi|_{Y \backslash X}} = \emptyset$, by Poincar\'e Recurrence Theorem, it follows that $\mu_{Y \backslash X} = 0$ and thus that $h_{\mu_{Y \backslash X}}\left(\psi|_{Y \backslash X}\right) = 0$. Hence
\begin{equation}
h_\mu\left(\psi\right) \leq h_{\mu_X}\left(\psi|_{X}\right)
\end{equation}
and thus it follows that
\begin{equation}
\sup_{\mu} h_\mu\left(\psi\right) = \sup_{\mu_X} h_{\mu_X}\left(\psi|_{X}\right)
\end{equation}
and the result follows from Theorem \ref{theo-variational-principle}.
\end{proof}

\section{Preliminaries on Lie theory}\label{Lie-Theory} 

In this section, we collect the main concepts and results, and prove some new ones, about Lie theory that are used in the final section. Given a Lie group $G$ with Lie algebra $\mathfrak{g}$, the connected component of the identity of $G$ is denoted by $G_0$. The center of $\mathfrak{g}$ is given by
\begin{equation}
\mathfrak{z}(\mathfrak{g}) = \{H \in \mathfrak{g} : [H,X] = 0, \mbox{ for all } X \in \mathfrak{g}\}
\end{equation}
which is an ideal of $\mathfrak{g}$ (see Lemma 11.1.1 of \cite{hilgert-neeb}), and the derived algebra $[\mathfrak{g},\mathfrak{g}]$ is the subalgebra generated by the subset
\begin{equation}
\{[X,Y] : X,Y \in \mathfrak{g}\}
\end{equation}
Given a subgroup $H \subset G$, the centralizer of a subgroup $H$ in $G$ is given by
\begin{equation}
Z(H,G) = \{g \in G : ghg^{-1} = h, \mbox{ for all } h \in H \}
\end{equation}
and the normalizer of a subgroup $H$ in $G$ is given by
\begin{equation}
N(H,G) = \{g \in G : gHg^{-1} = H \}
\end{equation}
The adjoint representation of $G$ is the map given by $\mathrm{Ad}(g) = d(C_g)_1$, where $C_g(h) = ghg^{-1}$ is the conjugation by $g \in G$. They are related by the following formula
\begin{equation}
\exp\left(\mathrm{Ad}(g)X\right) = g\exp(X)g^{-1}
\end{equation}
for all $g \in G$ and all $X \in \mathfrak{g}$, where $\exp : \mathfrak{g} \to G$ is the exponential map of $G$. The proof of the following result can be found in Section 9.5 of \cite{hilgert-neeb}, specially using Theorem 9.5.4 and Example 9.5.6 of \cite{hilgert-neeb}.

\begin{proposicao}\label{torus-discrete}
If $T$ is a torus, then $T \simeq \mathbb{R}^n/\mathbb{Z}^n$ and the group of the automorphisms of $T$ is isomorphic to the discrete Lie group $\mathrm{GL}(n,\mathbb{Z})$. 
\end{proposicao}

The proof following proposition can be found in Theorem 6.1.18 and Lemma 12.2.1 of \cite{hilgert-neeb}.

\begin{proposicao}\label{theorem-6-1-18-hilgert-neeb-lemma-12-2-1-hilgert-neeb}
If $\mathfrak{g}$ is a finite-dimensional Lie algebra and $H \in \mathfrak{g}$ is a regular element, then the centralizer $\frak{h}$ of $H$ in $\frak{g}$ is a Cartan subalgebra of $\frak{g}$. If $\mathfrak{g}$ is a compact Lie algebra, then a subalgebra $\frak{t} \subset \frak{g}$ is a Cartan subalgebra if and only if it is maximal abelian.
\end{proposicao}

The proof of the following result can be found in Theorem 4.5 of \cite{borel-mostow}.

\begin{proposicao}\label{theorem-4-5-borel-mostow}
A semi-simple (diagonalizable over the complex numbers) automorphism $\phi$ of a semi-simple Lie algebra $\frak{g}$ fixes a regular element $H \in \mathfrak{g}$.
\end{proposicao}

The proof of the following proposition can be found in Corollary 12.2.11 and Theorem 12.2.2 of \cite{hilgert-neeb}.

\begin{proposicao}\label{corollary-12-2-11-hilgert-neeb-theorem-12-2-2-hilgert-neeb}
Let $K$ be a compact connected group and $\frak{k}$ be its Lie algebra. If $T$ is a maximal torus, then $Z(T,G) = T$. A subalgebra $\frak{t} \subset \frak{k}$ is maximal abelian if and only if it is the Lie algebra of a maximal torus of G.
For two maximal tori $T$ and $Z$, there exists $h \in K$ such that $T = hZh^{-1}$.
\end{proposicao}

The proof of the following result can be found in item (iii) of Theorem 14.1.3 of \cite{hilgert-neeb}.

\begin{proposicao}\label{theorem-14-1-3-hilgert-neeb}
Let $H$ be a Lie group with finite number of connected components and $K \subset H$ be a maximal compact subgroup. Then, given $U \subset H$ a compact subgroup, there exists $h \in H_0$ such that $hUh^{-1} \subset K$.
\end{proposicao}

As far as we know, the following theorem is new. It is the analogous for compact disconnected Lie groups of a result that says that every element in a compact connected Lie group is conjugated to an element of a fixed maximal torus (See Proposition \ref{corollary-12-2-11-hilgert-neeb-theorem-12-2-2-hilgert-neeb} above).

\begin{teorema}\label{theo-conjugation-normalizer-torus}
Let $K$ be a compact Lie group and $T \subset K$ be a maximal torus. Then $N(T,K)_0 = T$ and, for every $g \in K$, there exists $h \in K_0$ such that $hgh^{-1} \in N(T,K)$.
\end{teorema}

\begin{proof}
Consider the map $\theta: N(T,K)_0 \to \mbox{Aut}(T)$, given by $\theta(g) = C_g|_T$. Since $\theta$ is continuous and $N(T,K)_0$ is connected, it follows that the image $\mbox{Im}(\theta)$ is also connected. Since $\mbox{Aut}(T)$ is discrete and $\mbox{Id}_T =  \theta(1) \in \mbox{Im}(\theta)$, it follows that $\mbox{Im}(\theta) = \{\mbox{Id}_T\}$. This implies that $N(T,K)_0 \subset Z(T,K)$. Since $N(T,K)_0 \subset K_0$, it follows that
\begin{equation}
N(T,K)_0 \subset Z(T,K) \cap K_0 = Z(T,K_0) = T
\end{equation}
where we used Proposition \ref{corollary-12-2-11-hilgert-neeb-theorem-12-2-2-hilgert-neeb} in the last equality. On the other hand, it is immediate that $T \subset N(T,K)_0$. 

For the second claim, let $g \in K$ and note that, since $K$ is compact, its Lie algebra $\frak{k}$ is reducible. Thus
\begin{equation}
\frak{k} = \frak{z}(\frak{k}) \oplus [\frak{k},\frak{k}]
\end{equation}
where $\frak{z}(\frak{k})$ is the center of $\frak{k}$ and $\frak{s} = [\frak{k},\frak{k}]$ is a semi-simple ideal of $\frak{k}$. Since $K$ is compact, it follows that $\mbox{Ad}(g)$ is a semi-simple automorphism of $\frak{s}$. By Proposition \ref{theorem-4-5-borel-mostow}, there exists a regular element $H \in \frak{s}$ such that $\mbox{Ad}(g)H = H$. By Proposition \ref{theorem-6-1-18-hilgert-neeb-lemma-12-2-1-hilgert-neeb}, the centralizer $\frak{h}$ of $H$ in $\frak{s}$ is a Cartan subalgebra of $\frak{s}$, such that $\mbox{Ad}(g)\frak{h} = \frak{h}$. Hence $\frak{z} = \frak{z}(\frak{k}) \oplus \frak{h}$ is a Cartan subalgebra of $\frak{k}$ such that $\mbox{Ad}(g)\frak{z} = \frak{z}$. It follows that
\begin{eqnarray}
gZg^{-1}
& = & g\langle \exp(\frak{z}) \rangle g^{-1} \\
& = & \langle \exp(\mbox{Ad}(g)\frak{z}) \rangle \nonumber \\
& = & \langle \exp(\frak{z}) \rangle \nonumber \\
& = & Z
\end{eqnarray}
By Propositions \ref{theorem-6-1-18-hilgert-neeb-lemma-12-2-1-hilgert-neeb} and \ref{corollary-12-2-11-hilgert-neeb-theorem-12-2-2-hilgert-neeb}, we have that $Z = \langle \exp(\frak{z}) \rangle$ is a maximal torus of $K_0$ and there exists $h \in K_0$ such that $T = hZh^{-1}$. It follows that
\begin{eqnarray}
h^{-1}Th
& = & Z \\
& = & gZg^{-1} \nonumber \\
& = & gh^{-1}Thg^{-1}
\end{eqnarray}
which implies that
\begin{equation}
hgh^{-1}T\left(hgh^{-1}\right)^{-1} = T
\end{equation}
showing that $hgh^{-1} \in N(T,K)$.
\end{proof}

\begin{corolario}\label{cor-conjugation-normalizer-torus}
Let $H$ be a Lie group with finite number of connected components, $K \subset H$ be a maximal compact subgroup, $T \subset K$ be a maximal torus, and $X$ be the union of all compact subgroups of $H$. Then
\begin{equation}
X = \bigcup_{g \in H_0} gN(T,K)g^{-1}
\end{equation}
\end{corolario}

\begin{proof}
By Proposition \ref{theorem-14-1-3-hilgert-neeb}, given $U \subset H$ a compact subgroup, there exists $h \in H_0$ such that $hUh^{-1} \subset K$. Thus, for each $u \in U$, we have that $huh^{-1} \in K$. By Theorem \ref{theo-conjugation-normalizer-torus}, there exists $k \in K_0$ such that $khuh^{-1}k^{-1} \in N(T,K)$, which implies that $u \in gN(T,K)g^{-1}$, where $g = h^{-1}k^{-1} \in H_0$.
\end{proof}

\section{Topological entropy of powers}\label{topological}

The first result of this main section determines the topological entropy of powers on compact Lie groups whose connected component of the identity is a torus.

\begin{proposicao}\label{prop-entropy-disjoint-tori}
Let $N$ be a Lie group with finite number of connected components such that $N_0 = T$ is a torus and $\psi : N \to N$ be the power map with exponent $k$. Then 
\begin{equation}
h\left(\psi\right) = \dim(T)\log(k)
\end{equation}
where $\dim(T)$ is the dimension of $T$.
\end{proposicao}

\begin{proof}
Consider the map $\Psi : N/T \to N/T$, given by $\Psi\left(\pi(g)\right) = \pi\left(\psi(g)\right)$, where $\pi : N \to N/T$ is the canonical projection. Since we have the following sequence
\begin{equation}
\cdots \subset \Psi^2\left(G/T\right) \subset \Psi\left(G/T\right) \subset G/T 
\end{equation}
and since $G/T$ is finite, there exists $j$ such that $\Psi\left(\Gamma\right) = \Gamma$ if $\Gamma = \Psi^j\left(G/T\right)$. Since $\Gamma$ is finite, it follows that $\Psi|_\Gamma$ is a bijection of $\Gamma$ and hence there exists a minimal natural number $n$ such that $\Psi^n|_\Gamma = \mathrm{Id}_\Gamma$. Furthermore, we have that $\mathcal{R}_\Psi \subset \Gamma$, which implies that $\mathcal{R}_\psi \subset \pi^{-1}\left(\Gamma\right)$, since $\pi\left(\mathcal{R}_\psi\right) \subset \mathcal{R}_\Psi$ by the definitions and by the continuity of the maps. Since $\pi^{-1}\left(\Gamma\right)$ is closed, it follows that $\overline{\mathcal{R}_{\psi}} \subset \pi^{-1}\left(\Gamma\right)$.
We have that
\begin{eqnarray}
h\left(\psi\right)
& = & h\left(\psi|_{\overline{\mathcal{R}_{\psi}}}\right) \nonumber \\
& = & h\left(\psi|_{\pi^{-1}\left(\Gamma\right)}\right) \nonumber \\
& = & \frac{1}{n}h\left(\psi^n|_{\pi^{-1}\left(\Gamma\right)}\right)
\end{eqnarray}
If $\Delta$ is any subset of $N$ such that $\pi$ is a bijection between $\Delta$ and $\Gamma$, then
\begin{equation}
\pi^{-1}\left(\Gamma\right) =  \bigcup_{g \in \Delta} gT
\end{equation}
which is a disjoint union. Since $\Psi^n|_\Gamma = \mathrm{Id}_\Gamma$, it follows that $\psi^n\left(gT\right) = gT$ for $g \in \Delta$, which implies, by Proposition \ref{prop-disjoint-union}, that
\begin{equation}\label{eq-entropy-disjoint-tori-1}
h\left(\psi\right) = \frac{1}{n} \max_{g \in \Delta} h\left(\psi^n|_{gT}\right)
\end{equation}
On the other hand, we have that
\begin{eqnarray}
\psi^n\left(gt\right)
& = & \left(gt\right)^{k^n} \nonumber \\
& = & gtgtgt\cdots gtgtgt \nonumber \\
& = & g^{k^n}g^{1-k^n}tg^{k^n-1}g^{2-k^n}tg^{k^n-2}g^{3-k^n}t\cdots g^3g^{-2}tg^2g^{-1}tgt \nonumber \\
& = & g^{k^n}\phi^{k^n-1}(t)\phi^{k^n-2}(t)\cdots \phi^2(t)\phi(t)t \nonumber \\
\end{eqnarray}
where $\phi : T \to T$ is the automorphism given by $\phi(t) = g^{-1}tg$. Since $\psi^n\left(gt\right) \in gT$, it follows that $g^{k^n} \in gT$ and hence $g^{k^n-1} \in T$. Thus it follows that
\begin{equation}\label{eq-entropy-disjoint-tori-2}
g^{-1}\psi^n\left(gt\right) = g^{k^n-1}\varphi(t)
\end{equation}
where 
\begin{equation}
\varphi(t) = \phi^{k^n-1}(t)\phi^{k^n-2}(t)\cdots \phi^2(t)\phi(t)t
\end{equation}
is an endomorphism of $T$. In fact, we have that
\begin{eqnarray}
\varphi\left(t_1t_2\right)
& = & \phi^{k^n-1}(t_1t_2)\cdots \phi^2(t_1t_2)\phi(t_1t_2)t_1t_2 \nonumber \\
& = & \phi^{k^n-1}(t_1)\phi^{k^n-1}(t_2)\cdots \phi^2(t_1)\phi^2(t_2)\phi(t_1)\phi(t_2)t_1t_2 \nonumber \\
& = & \phi^{k^n-1}(t_1)\cdots \phi^2(t_1)\phi(t_1)t_1\phi^{k^n-1}(t_2)\cdots \phi^2(t_2)\phi(t_2)t_2 \nonumber \\
& = & \varphi\left(t_1\right)\varphi\left(t_2\right)
\end{eqnarray}
where we used that $T$ is abelian. By the conjugation given in equation (\ref{eq-entropy-disjoint-tori-2}) and by Proposition \ref{corollaries-11-16-bowen-endomorphis}, it follows that
\begin{eqnarray}\label{eq-entropy-disjoint-tori-3}
h\left(\psi^n|_{gT}\right)
& = & h\left(L_{g^{k^n-1}} \circ \varphi\right) \nonumber \\
& = & h\left(\varphi\right) \nonumber \\
& = & \sum_{\lambda} \log\left|\lambda\right|
\end{eqnarray}
the summation is taken over all eigenvalues $\lambda$ of $d\varphi_1$ satisfying $|\lambda| > 1$. Since 
\begin{equation}
d\varphi_1 = d\phi^{k^n-1}_1 + \cdots + d\phi^2_1 + d\phi_1 + \mathrm{Id}
\end{equation}
it follows that
\begin{equation}
\lambda = \rho^{k^n-1} + \cdots + \rho^2 + \rho + 1
\end{equation}
where $\rho$ is an eigenvalue of $d\phi_1$. We claim that $\phi$ has finite order. In fact, consider the homomorphism $\theta : N \to \mathrm{Aut}(T)$, given by $\theta(g) = C_g|_T$. Since $\theta$ is continuous, $T$ is connected, $\mathrm{Aut}(T)$ is discrete and $\theta(1) = \mbox{Id}_T$, it follows that $T \subset \ker(\theta)$. Thus we can consider the induced homomorphism $\Theta : N/T \to \mathrm{Aut}(T)$, given by $\Theta(\pi(g)) = \theta(g)$. Since $N/T$ is finite, it follows that $\mathrm{Im}(\Theta)$ is finite, which implies that $\Theta(\pi(g)) = \theta(g) = \phi$ has finite order. It follows that $|\rho| = 1$ and thus
\begin{equation}
|\lambda| \leq |\rho|^{k^n-1} + \cdots + |\rho|^2 + |\rho| + 1 = k^n
\end{equation}
By equation (\ref{eq-entropy-disjoint-tori-3}), it follows that
\begin{equation}
h\left(\psi^n|_{gT}\right) \leq \dim(T)\log(k^n)
\end{equation}
By equation (\ref{eq-entropy-disjoint-tori-1}) and since $\theta(1) = \mbox{Id}_T$, it follows that
\begin{equation}
h\left(\psi\right) = \frac{1}{n} \max_{g \in \Delta} h\left(\psi^n|_{gT}\right) = \dim(T)\log(k)
\end{equation}
completing the proof.
\end{proof}

The next result shows that the recurrent set of powers on a Lie group $G$ lies inside the union of all compact subgroups of $G$.

\begin{proposicao}\label{prop-recurrent-set-compact-subgroups}
Let $G$ be a Lie group and $\psi : G \to G$ be the power map with exponent $k$. Then $\mathcal{R}_{\psi} \subset X$, where $X$ is the union of all compact subgroups of $G$.
\end{proposicao}

\begin{proof}
If $g \in \mathcal{R}_{\psi}$, then there exists $n_j \to \infty$ such that $\psi^{n_j}(g) \to g$. Defining $m_j = k^{n_j} - 1$, it follows that 
\begin{equation}
g^{m_j} = g^{k^{n_j} - 1} = \psi^{n_j}(g)g^{-1} \to 1
\end{equation}
since $\psi^{n_j}(g) = g^{k^{n_j}}$. Now let $A$ be the closure of the subgroup generated by $g$. It follows that $A$ is closed abelian subgroup of $G$. Hence $A_0 = V \times T$, where $V$ is an euclidean space and $T$ is the maximal torus of $A$. Since $g^{m_j} \to 1$, there exists $N$ such that $g^N \in A_0$. Thus $g^N = \exp(Y+Z)$, with $Y \in \frak{v}$ and $Z \in \frak{t}$, where $\frak{v}$ and $\frak{t}$ are the Lie algebras of respectively $V$ and $T$. Writing $m_j = Nq_j + r_j$, where $0 \leq r_j < N$, it follows that $q_j \to \infty$. Thus
\begin{equation}
g^{m_j} = \left(g^N\right)^{q_j}g^{r_j} = \exp\left(q_jY\right)\exp\left(q_jZ\right)g^{r_j}
\end{equation}
We have that 
\begin{equation}
\exp\left(q_jZ\right)g^{r_j} \in \bigcup_{r = 0}^{N-1} (0 \times T)g^r
\end{equation}
which is compact. Thus there exists a subsequence $\exp\left(q_{j_l}Z\right)g^{r_{j_l}} \to h$ and hence
\begin{equation}
\exp\left(q_{j_l}Y\right) = g^{m_{j_l}}\left(\exp\left(q_{j_l}Z\right)g^{r_{j_l}}\right)^{-1} \to h^{-1}
\end{equation}
Since $q_{j_l} \to \infty$, this implies that $Y = 0$, which implies that $g^N = \exp(Z) \in 0 \times T$ and that $g^{-N} = \exp(-Z) \in 0 \times T$. Hence 
\begin{equation}
A \subset \bigcup_{r = -N+1}^{N-1} (0 \times T)g^r
\end{equation}
showing that $A$ is compact.
\end{proof}

The next result shows that the computation of the topological entropy of powers on Lie groups with finite number of connected components reduces to the computation on Lie groups whose connected component of the identity is a torus.

\begin{proposicao}\label{prop-entropy-finite-components}
Let $H$ be a Lie group with finite number of connected components, $\psi : H \to H$ be the power map with exponent $k$, $K \subset H$ be a maximal compact subgroup, $X$ be the union of all compact subgroups of $H$, and
\begin{equation}
H = \bigcup_{l=0}^\infty B_l
\end{equation}
where each $B_l$ is a compact subset and $B_l \subset B_{l+1}$ for each $l$. Then
\begin{equation}
X = \bigcup_{l=0}^\infty C_l
\end{equation}
where
\begin{equation}
C_l = \bigcup_{g \in B_l} gN(T,K)g^{-1}
\end{equation}
is a $\psi$-invariant compact subset and $C_l \subset C_{l+1}$ for each $l$, and 
\begin{equation}
h\left(\psi\right) = h\left(\psi|_{N(T,K)}\right)
\end{equation}
\end{proposicao}

\begin{proof}
Since $c : B_l \times N(T,K) \to C_l$, given by $c(g,t) = gtg^{-1}$, is a continuous and surjective map and since $B_l \times N(T,K)$ is compact, it follows that $C_l$ is compact. We also have that $C_l \subset C_{l+1}$ for each $l$, since $B_l \subset B_{l+1}$ for each $l$, and that $C_l$ is $\psi$-invariant, since
\begin{eqnarray}
\psi\left(C_l\right)
& = & \bigcup_{g \in B_l} \psi\left(gN(T,K)g^{-1}\right) \\
& = & \bigcup_{g \in B_l} g\psi\left(N(T,K)\right)g^{-1} \nonumber \\
& = & \bigcup_{g \in B_l} gN(T,K)g^{-1} \nonumber \\
& = & C_l
\end{eqnarray}
By Corollary \ref{cor-conjugation-normalizer-torus}, we have that
\begin{equation}
X = \bigcup_{g \in H} gN(T,K)g^{-1}
\end{equation}
which implies that
\begin{equation}
X = \bigcup_{l=0}^\infty \bigcup_{g \in B_l} gN(T,K)g^{-1}
\end{equation}
since
\begin{equation}
H = \bigcup_{l=0}^\infty B_l
\end{equation}
By Proposition \ref{theorem-3-1-caldas-patrao:entropia}, Proposition \ref{prop-recurrent-set-compact-subgroups}, and Corollary \ref{cor-variational-principle}, it follows that
\begin{equation}\label{eq-entropy-finite-components}
h\left(\psi\right) = \sup_{\mu} h_\mu\left(\psi\right) = \lim_{l \to \infty} h\left(\psi|_{C_l}\right)
\end{equation}
We also have that
\begin{equation}
\psi\left(c(g,t)\right) = c\left(g,\psi(t)\right)
\end{equation}
which means that the map $c$ is a semi-conjugation between the map $\psi|_{C_l}$ and the map $\Psi : B_l \times N(T,K) \to B_l \times N(T,K)$, given by $\Psi\left(g,t\right) = \left(g,\psi(t)\right)$. Since $c$ is a continuous and surjective map and since $B_l \times N(T,K)$ is compact, it follows that
\begin{eqnarray}
h\left(\psi|_{N(T,K)}\right)
& \leq & h\left(\psi|_{C_l}\right) \nonumber \\
& \leq & h\left(\Psi\right) \nonumber \\
& = & h\left(\psi|_{N(T,K)}\right)
\end{eqnarray}
showing that
\begin{equation}
h\left(\psi|_{C_l}\right) = h\left(\psi|_{N(T,K)}\right)
\end{equation}
and thus that
\begin{equation}
h\left(\psi\right) = h\left(\psi|_{N(T,K)}\right)
\end{equation}
where we used equation (\ref{eq-entropy-finite-components}).
\end{proof}

The next result shows how to reduce the general case to the previous one.

\begin{proposicao}\label{prop-entropy-powers}
Let $G$ be a Lie group and $\psi : G \to G$ be the power map with exponent $k$. Then 
\begin{equation}
h\left(\psi\right) = \sup_{H} h\left(\psi|_H\right)
\end{equation}
where the supremum is taken over all open subgroups $H$ of $G$ with finite number of connected components. 
\end{proposicao}

\begin{proof}
Since $G$ has a countable number of connected components, it follows that the family of all open subgroups $H$ of $G$ with finite number of connected components is also countable, given by $\{H_0, H_1, H_2, \ldots, H_n, \ldots\}$. We also have that
\begin{equation}
X_G = \bigcup_{n=0}^\infty X_{H_n}
\end{equation}
where $X_G$ and $X_{H_n}$ are the union of all compact subgroups of respectively $G$ and $H_n$. By Proposition \ref{prop-entropy-finite-components}, we have that
\begin{equation}
X_{H_n} = \bigcup_{l=0}^\infty C_l^n
\end{equation}
where each $C_l^n$ is a $\psi$-invariant compact subset and $C_l^n \subset C_{l+1}^n$ for each $l$. Hence
\begin{equation}
X_G = \bigcup_{l=0}^\infty C_l
\end{equation}
where 
\begin{equation}
C_l = \bigcup_{n=0}^l C_l^n
\end{equation}
is a $\psi$-invariant compact subset and $C_l \subset C_{l+1}$ for each $l$, since 
\begin{equation}
C_l^n \subset C_{\max\{l,n\}}^n \subset C_{\max\{l,n\}}
\end{equation}
Now, for each $l$, we can write
\begin{equation}
\bigcup_{n=0}^l H_n = \bigcup_{i=0}^j G_i
\end{equation}
where $\{G_0, G_1, \ldots, G_j\}$ is a family of disjoint connected components of $G$. There exists a positive integer $m$ such that $\psi^m\left(G_i\right) = G_i$ for each $i$. Thus
\begin{eqnarray}
h\left(\psi^m|_{\bigcup_{n=0}^l H_n}\right)
& = & h\left(\psi^m|_{\bigcup_{i=0}^j G_i}\right) \nonumber \\
& = & \max_{i = 0,\ldots,j} h\left(\psi^m|_{G_i}\right) \nonumber \\
& \leq & \max_{n = 0,\ldots,l} h\left(\psi^m|_{H_n}\right)
\end{eqnarray}
where we used Proposition \ref{prop-disjoint-union} in the second equality. Hence
\begin{eqnarray}
h\left(\psi|_{\bigcup_{n=0}^l H_n}\right)
& = & \frac{1}{m}h\left(\psi^m|_{\bigcup_{n=0}^l H_n}\right) \nonumber \\
& \leq & \frac{1}{m}\max_{n = 0,\ldots,l} h\left(\psi^m|_{H_n}\right) \nonumber \\
& = & \max_{n = 0,\ldots,l} h\left(\psi|_{H_n}\right) \nonumber \\
& \leq & \sup_{H} h\left(\psi|_H\right)
\end{eqnarray}
Therefore
\begin{eqnarray}
h\left(\psi\right)
& = & \lim_{l \to \infty} h\left(\psi|_{C_l}\right) \nonumber \\
& \leq & \lim_{l \to \infty} h\left(\psi|_{\bigcup_{n=0}^l H_n}\right) \nonumber \\
& \leq & \sup_{H} h\left(\psi|_H\right) \nonumber \\
& \leq & h\left(\psi\right)
\end{eqnarray}
where we used Proposition \ref{prop-recurrent-set-compact-subgroups} and Corollary \ref{cor-variational-principle} in the first equality.
\end{proof}

Now it is immediate the main theorem of the paper.

\begin{teorema}\label{theo-entropy-powers}
Let $G$ be a Lie group, $T$ be a maximal torus of $G$, and $\psi : G \to G$ be the power map with exponent $k$. Then
\begin{equation}
h\left(\psi\right) = \dim(T)\log(k)
\end{equation}
where $\dim(T)$ is the dimension of $T$.
\end{teorema}

\begin{proof}
Immediate consequence of Propositions \ref{prop-entropy-disjoint-tori}, \ref{prop-entropy-finite-components}, and \ref{prop-entropy-powers}.
\end{proof}

\end{document}